\documentclass[reqno]{amsart}
\usepackage{amsmath, amsthm, amssymb, amstext}

\usepackage{hyperref,xcolor}
\hypersetup{
  pdfborder={0 0 0},
  colorlinks,
}
\usepackage{enumitem}
\newtheorem{theorem}{Theorem}

\newtheorem{lemma}[theorem]{Lemma}
\newtheorem{proposition}[theorem]{Proposition}

\newtheorem{definition}[theorem]{Definition}

\newtheorem{remark}[theorem]{Remark}

\newcommand{\N}{\mathbb{N}}

\newcommand{\R}{\mathbb{R}}

\newcommand{\eps}{\varepsilon}
\newcommand{\ph}{\varphi}

\newcommand{\weak}{\rightharpoonup}

\newcommand{\cprime}{$'$}

\numberwithin{theorem}{section}
\numberwithin{equation}{section}

\title[Nonlinear problems in $\R^N$ involving the fractional Laplacian]{Bounded solutions to nonlinear problems in $\R^N$ involving the fractional Laplacian depending on parameters}

\author[S.\,El Manouni]{Said El Manouni}
\address[S.\,El Manouni]{Technische Universit\"{a}t Berlin, Institut f\"{u}r Mathematik, Stra\ss e des 17.\,Juni 136, 10623 Berlin, Germany}
\email{manouni@math.tu-berlin.de}

\author[H.\,Hajaiej]{Hichem Hajaiej}
\address[H.\,Hajaiej]{New York University, Shanghai, 1555 Centry Avenue, Pudong New District, Shanghai 200122, China}
\email{hichem.hajaiej@gmail.com}

\author[P.\,Winkert]{Patrick Winkert}
\address[P.\,Winkert]{Technische Universit\"{a}t Berlin, Institut f\"{u}r Mathematik, Stra\ss e des 17.\,Juni 136, 10623 Berlin, Germany}
\email{winkert@math.tu-berlin.de}

\subjclass[2010]{35R11, 35J20, 35J60}
\keywords{Fractional Laplacian, Nonlocal eigenvalue problems, unbounded domains, existence and regularity, multiplicity results, Ricceri's principle}

\begin{document}

\begin{abstract}
    The main goal of this paper is the study of two kinds of nonlinear problems depending on parameters in unbounded domains. Using a nonstandard variational approach, we first prove the existence of bounded solutions for nonlinear eigenvalue problems involving the fractional Laplace operator $(-\Delta )^s$ and nonlinearities that have subcritical growth. In the second part, based on a variational principle of Ricceri \cite{Ricceri-2009}, we study a fractional nonlinear problem with two parameters and prove the existence of multiple solutions.
\end{abstract}

\maketitle

\section{Introduction}

In this paper, we are interested in the study of two nonlocal problems involving an integro-differential operator of fractional type. More precisely, the fractional Laplacian in $\R^N $ as a nonlocal generalization of the Laplace operator denoted by $( -\Delta)^s$ and defined pointwise for $x\in \R^N$ up to normalization factors by
\begin{equation}\label{e1.1}
    (-\Delta)^s u(x)=-\frac{1}{2}\int_{\R^N}\frac{u(x+y)+u(x-y)-2u(x)}{|y|^{N+2s}}dy,
\end{equation}

takes center in our considerations in terms of two nonlocal problems involving one, respectively two positive parameters. Let us mention that in recent years much attention has been devoted to the study of nonlocal problems of elliptic type and a lot of papers have appeared both in bounded and unbounded domains, see, for example Autuori-Pucci \cite{Autuori-Pucci-2013}, Molica Bisci \cite{Molica-Bisci-2014}, Molica Bisci-R{\u{a}}dulescu \cite{Molica-Bisci-Radulescu-2015}, Lehrer-Maia-Squassina \cite{Lehrer-Maia-Squassina-2015}, Servadei-Valdinoci \cite{Servadei-Valdinoci-2012} the references therein.

Nonlocal operators such as $(-\Delta )^s$  naturally arise in a quite natural way in many different contexts, both for pure mathematical research and in view of concrete real-world applications, such as, for instance, the thin obstacle problem, optimization, finance, phase transitions phenomena, stratified materials, anomalous diffusion, population dynamics and game theory, flame propagation, conservation laws, continuum mechanics, quasi-geostrophic flows, multiple scattering, minimal surfaces, etc. For more details, we refer to Caffarelli \cite{Caffarelli-2009}, Caffarelli-Vazquez \cite{Caffarelli-Vazquez-2011}, Chang-Wang \cite{Chang-Wang-2013}, Cheng \cite{Cheng-2012}, Felmer-Quaas-Tan \cite{Felmer-Quaas-Tan-2012}, Hajaiej \cite{Hajaiej-2013}, Tan-Wang-Yang \cite{Tan-Wang-Yang-2012} and the references therein.

Dealing with such nonlocal problems involving $(-\Delta )^s$ operator requires a particular functional framework. We refer the reader to Di Nezza-Palatucci-Valdinoci \cite{Di-Nezza-Palatucci-Valdinoci-2012} and to the references included for a selfcontained overview of the basic properties of fractional Sobolev spaces.


The purpose of this paper is the study of two nonlocal problems by using two different approaches. In the first part, we consider the eigenvalue problem for the fractional Laplacian in $\R^N$ and using a nonstandard variational procedure we show the existence of at least one solution corresponding to some positive parameter. Additionally, we establish some regularity results for the solution obtaind in terms of global $L^\infty$-estimates by constructing an iteration scheme to bound the maximal norm of the solution following Moser's iteration (see, for example Dr{\'a}bek-Kufner-Nicolosi \cite{Drabek-Kufner-Nicolosi-1997}). The second part concerns a nonlinear problem with two parameters involving the fractional Laplacian. The existence and multiplicity results are proved by applying a variational principle of Ricceri \cite{Ricceri-2009}. Let us note that since the appearance of the abstract result proved by Ricceri in \cite{Ricceri-2000} and its revisited note established in \cite{Ricceri-2009-b} dealing with variational equations in both local and nonlocal settings, they have been widely investigated and have intensively been applied for the study of the existence of multiple nontrivial solutions and in recent years a lot of papers has been appeared in the scalar case and systems of elliptic equations. In \cite{Ricceri-2009}, Ricceri obtained a general three critical points theorem, that has been applied for a class of local integro-differential equations involving a large class of nonlinearities. We will not mention such applications here since the reader can easily access to such literature in this regard. In our context, we are interested in a class of nonlocal integro-differential equations involving two parameters and nonlinearities with subcritical growth in $\R^N$ and prove multiple nontrivial solutions by making use of the variational principle
of Ricceri \cite{Ricceri-2009}.

\section {Preliminaries}
Throughout this paper and before defining the natural framework for such two nonlocal problems involving the fractional Laplace operator \eqref{e1.1}, namely, a fractional Sobolev-type space, we will denote by $\|\cdot\|_p$ the Lebesgue norm in $L^p\left(\R^N\right)$  and $\mathcal{C}_0^\infty\left(\R^N\right)$ will be the space of all functions with compact support in $\R^N$ with continuous derivatives of arbitrary order.

Let us introduce the function space $D^s\left(\R^N\right)$ as the completion of $\mathcal{C}_0^{\infty}({\R}^N)$ with respect to the Gagliardo norm
\begin{align*}
    [u]_s= \left(\ \iint\limits_{\R^{2N}}\frac{|u(x)-u(y)|^2}{|x-y|^{N+2s}}dxdy\right)^{\frac{1}{2}}.
\end{align*}
It can be shown that
\begin{align*}
    D^s\left(\R^N\right)=\left\{ u \in L^{2_s^*}\left(\R^N\right): \frac{|u(x)-u(y)|}{|x-y|^{N/2+s}}\in L^2 \left(\R^N \times \R^N\right) \right\}
\end{align*}
and that
\begin{equation} \label{e2.1}
    \|u\|_{2_s^*}\leq C_{2_s^*} [u]_s \quad \text{for all }u\in D^s\left(\R^N\right),
\end{equation}
where $2_s^*=\frac{2N}{N-2s}$ and $C_{2_s^*}^2=c(N)\frac{s(1-s)}{(N-2s)}$, cf.\,Maz{\cprime}ya-Shaposhnikova \cite[Theorem 1]{Mazya-Shaposhnikova-2002} (see also Bourgain-Brezis-Mironescu \cite[Theorem 1]{Bourgain-Brezis-Mironescu-2002}). Moreover, the space $D^s\left(\R^N\right)$ is a reflexive Banach space continuously embedded into $L^{2_s^*}\left(\R^N\right)$.

The following two lemmas will be used in later considerations, see Brasco-Parini \cite[Lemmas A.1 and A.2]{Brasco-Parini-2014}.
\begin{lemma}\label{Lemma1}
    Let $1<p<\infty$ and $f:\R\to \R$ be a convex function, then
    \begin{align*}
	& |a-b|^{p-2}(a-b)\left[A|f'(a)|^{p-2}f'(a)-B|f'(b)|^{p-2}f'(b)\right]\\
	& \geq |f(a)-f(b)|^{p-2}(f(a)-f(b))(A-B),
    \end{align*}
    for every $a, b\in \R$ and every $A, B\geq 0$.
\end{lemma}

\begin{lemma}\label{Lemma2}
    Let $1<p<\infty$ and $g:\R\to \R$ be an increasing function. Define
    \begin{align*}
	G(t)=\int_0^tg'(s)^{\frac{1}{p}}\,ds,\;\;t\in\R.
    \end{align*}
    then we have
    \begin{align*}
	|a-b|^{p-2}(a-b)(g(a)-g(b))\geq |G(a)-G(b)|^{p}.
    \end{align*}
\end{lemma}

The subsequent Lemma can be found in Autuori-Pucci \cite{Autuori-Pucci-2013}.
\begin{lemma}\label{Lemma3}
    For any $R>0$ and $p\in \left[1, 2_s^*\right)$ the embedding $D^s(B_R)\hookrightarrow L^p(B_R)$ is compact, where $B_R$ denotes the ball in $\R^N$ of center zero and radius $R>0$.
\end{lemma}

Now, let $X$ be a real Banach space. We denote by $\mathcal{W}_X$ the class of all functionals $\Phi: X \to \R$ possessing the following property: if $(u_n)_{n \geq 1}$ is a sequence in $X$ converging weakly to $u\in X$ and $\liminf_{n\to \infty} \Phi(u_n)\leq\Phi(u)$, then $(u_n)_{n \geq 1}$ has a subsequence converging strongly to $u$.

The following abstract critical point theorem is due to Ricceri \cite[Theorem 2]{Ricceri-2009}.
\begin{theorem}\label{ThmRic}
    Let $X$ be a separable and reflexive real Banach space, $\Phi: X \to \R$ is a coercive, sequentially weakly lower semicontinuous $C^1$-functional belonging to $\mathcal{W}_X$ which is bounded on each bounded subset of $X$ and whose derivative admits a continuous inverse on $X^*$ and  $J: X \to \R $ is a $C^1$-functional with compact derivative. Assume that $\Phi$ has a strict local minimum $x_0$ with $\Phi (x_0)=J(x_0)=0$. Finally, set
    \begin{align*}
	& \alpha = \max \left\{0, \limsup_{\|x\|\to +\infty} \frac{J(x)}{\Phi (x)},\limsup_{x\to x_0} \frac{J(x)}{\Phi (x)}\right\},
	&\beta=\sup_{x\in \Phi^{-1}\left(\left]0, +\infty\right[\right)}\frac{J(x)}{\Phi (x)},
    \end{align*}
    and assume that $\alpha < \beta$. Then, for each compact interval $[a, b]\subset \left ]\frac{1}{\beta}, \frac {1}{\alpha}\right[$ (with the conventions $\frac{1}{0}=+\infty, \frac{1}{+\infty}=0$), there exists $r>0$ with the following property: for every $\lambda \in [a, b]$,  and every $C^1$-functional $\Psi: X \to \R $ with compact derivative, there exists $\sigma > 0$ such that for each $\mu \in [0, \sigma]$, the equation
    \begin{align*}
	\Phi^{'}(x) = \lambda J^{'}(x) + \mu \Psi^{'}(x)
    \end{align*}
    has at least three solutions whose norms are less than $r$.
\end{theorem}

\section{A nonlinear eigenvalue problem}

In this section, we consider the following nonlinear eigenvalue problem with a parameter $\lambda>0$
\begin{equation}\label{e3.1}
    (-\Delta)^s u = \lambda w(x)|u|^{q-2}u \quad\text{in }{\R}^N,
\end{equation}
where $0<s<1, 2s< N, 2<q<2_s^*$ and $(-\Delta)^s$ is the fractional Laplace operator given in \eqref{e1.1}. We suppose the following assumptions on the indefinite weight function $w:\R^N\to\R$.
\begin{enumerate}
    \item[(W$_1$)]
	there exists an open subset $\Omega \neq \emptyset$ of $\R^N$ such that $w(x)>0$ a.e.\,in $\Omega$;
    \item[(W$_2$)]
	$w\in L^{\tau}\left(\R^N\right)$ with $\tau\in\displaystyle \left[\frac{2_s^*}{2_s^*-q}, \frac{\nu}{\nu-1}\frac{2_s^*}{2_s^*-q}\right]$ for some $q<\nu < 2_s^*=\frac{2N}{N-2s}$.
\end{enumerate}

\begin{definition}
    We say that $u \in D^s\left(\R^N\right)$ is a weak solution of \eqref{e3.1} if there exists  $ \lambda >0$ such that
    \begin{equation}\label{e3.2}
    \iint\limits_{\R^{2N}}\frac{[u(x)-u(y)][\varphi(x)-\varphi(y)]}{|x-y|^{N+2s}}dxdy=\lambda \int\limits_{\R^N} w(x)|u|^{q-2}u\varphi\,dx,
    \end{equation}
    for all $\varphi \in D^s\left(\R^N\right)$.
\end{definition}

Our first result concerns the existence of at least one solution to problem \eqref{e3.1}.

\begin{theorem}\label{Thm1}
    Suppose that $ 0<s <1$, $2s<N$, $2<q<2_s^*$, (W$_1$) and (W$_2$) are satisfied. Then \eqref{e3.1} possesses at least one nontrivial weak solution  $ u\in D^s\left(\R^N\right)$ with
    \begin{align*}
	  \lambda = \frac {q \left( [u ]_s^2 +  [u ]_s^{2_s^*}\right)}{\left(2 +  2_s^*[ u]_s^{2_s^*-2}\right) \displaystyle\int_{\R^N} w(x)|u|^{q} \,dx}.
    \end{align*}
\end{theorem}
\begin{proof}
    Let us consider the functional $J:D^s\left(\R^N\right)\setminus \{0\} \to \R$  defined by
    \begin{align*}
    J(u) = \frac {\displaystyle\int_{\R^N} w(x)|u|^{q}\,dx}{\displaystyle[u]_s^2 + [u]_s^{2_s^*}},
    \end{align*}
    which is well defined and bounded. Indeed, in view of (W$_2$) and \eqref{e2.1}, we get
    \begin{align*}
	\int_{\R^N} w(x)|u|^{q}\,dx  \leq \| w \|_{\frac{2_s^*}{2_s^*-q}} \| u \|_{2_s^*}^q \leq  c_1 [u]_{s}^{q},
    \end{align*}
    where $ c_1 = C_{2_s^*}^q\| w \|_{\frac{2_s^*}{2_s^*-q}} < \infty $. On the other hand, since $2 < q <2_s^*$ we have
    \begin{align*}
	[u]_{s}^{q} \leq  [u]_{s}^{2} + [u]_{s}^{2_s^*} \quad \text{for all }u\in D^s\left(\R^N\right).
    \end{align*}
    Hence, it follows that
    \begin{align*}
	J(u) \leq c_1\quad\text{for all }u\in D^s\left(\R^N\right)\setminus \{0\}.
    \end{align*}

    Now let
    \begin{align*}
	\mathcal{S}: =\sup_{0 \neq u\in D^s\left(\R^N\right)}J(u).
    \end{align*}
    We choose $ \varphi_0 \in \mathcal{C} _0^\infty\left(\R^N\right) $ such that
    \begin{align*}
	\sup _{x\in {\R}^N} \varphi _0 (x) > 0,
    \end{align*}
    which gives $\mathcal{S} > \mathcal{S} _0$ with $\mathcal{S} _0 = \frac {1}{2}J(\varphi _0) > 0$. Let  $ (u _k)_{k \geq 1} \subset D^s\left(\R^N\right) $ with $ u _k\neq 0 $ be  such that
    \begin{align*}
	J(u _k) \to \mathcal{S} \quad \text{as }k\to +\infty.
    \end{align*}
    Since $ \mathcal{S} > S _0 $ and $  J(u^+) \geq J(u), $ one can choose
    \begin{align*}
	u _k \geq 0 \quad\text{and}\quad J(u _k) \geq S _0 \quad \text{for all } k \in \N.
    \end{align*}
    Then
    \begin{align*}
	S _0 \left( [u _k]_{s}^{2} +  [u _k]_{s}^{2_s^*}\right) \leq c_1 [u _k ]_{s}^{q} \quad\text{for all } k \in \N.
    \end{align*}
    Hence there exist constants $ 0 < \Lambda_1 < \Lambda_2 < \infty $ such that
    \begin{align*}
	\Lambda_1 \leq [u _k ]_{s} \leq  \Lambda_2
    \end{align*}
    holds for all $ k \in {\N}$. Then, one can find a subsequence still denoted by  $ (u _k)_{k \geq 1} $ and an element  $u \in D^s\left(\R^N\right) $ such that
    \begin{align*}
	u_k \weak u \quad\text{in }D^s\left(\R^N\right).
    \end{align*}
    Taking Lemma \ref{Lemma3} into account, we obtain that
    \begin{align*}
	D^s\left(B_R\right) \hookrightarrow L^p(B_R)
    \end{align*}
    is compact which implies that
    \begin{align*}
	u_k \to u \quad\text{in }  L^p\left(B_R\right)\quad \text{for all }p\in \left[1, 2_s^*\right).
    \end{align*}
    Since $\R^N = \cup_{R>0}B_R$, we deduce that $u \geq 0$ a.e.\,in $\R^N$. By using H\"older's inequality and \eqref{e2.1}, we have for all $ R>0 $ and  $ k \in {\N}$,
    \begin{align*}
	\left|\int _{|x| \geq R}w|u_k(x)|^q dx\right|
	& \leq  \left(\int _{|x| \geq R} |w(x)|^{\frac{2_s^*}{2_s^*-q}}\,dx\right)^{\frac {2_s^*-q}{2_s^*}}\left(\int _{|x| \geq R}|u _k|^{2_s^*}\,dx\right)^{\frac {q}{2_s^*}}\\
	& \leq  C_{2_s^*}^{q}\left(\int _{|x|\geq R} |w(x)|^{\frac{2_s^*}{2_s^*-q}}\,dx\right)^{\frac {2_s^*-2s}{2_s^*}} [u _k ]_s^{q}\\
	& \leq  c_2 \left(\int _{|x| \geq R} |w(x)|^{\frac{2_s^*}{2_s^*-q}}\,dx\right)^{\frac {2_s^*-q}{2_s^*}},
    \end{align*}
    where  $c_2$ is a positive constant independent of $R$ and $k$. In addition, we have
    \begin{align*}
	\left|\int _{|x| \geq R}w|u(x)|^q  \,dx\right|\leq   c_3  \left(\int _{|x| \geq R} |w(x)|^{\frac{2_s^*}{2_s^*-q}}\,dx\right)^{\frac {2_s^*-q}{2_s^*}},
    \end{align*}
    for some constant $c_3>0$ independent of $R$. This implies that for each $ \varepsilon > 0$ there exists $ R _\varepsilon > 0 $ such that
    \begin{equation} \label{e3.3}
	\left|\int _{|x| \geq R_\varepsilon}w|u_k(x)|^q\,dx\right| \leq \varepsilon \quad \text{and}\quad \left|\int _{|x| \geq R_\varepsilon}w|u(x)|^q\,dx\right| \leq  \varepsilon,
    \end{equation}
    hold for all $k \in \N$.

    On the other hand, applying Young's inequality, we get
    \begin{align*}
	\left(w|t|^{q-1}\right)^{\frac{\nu}{\nu-1}}
	& =  w(x)^\frac{\nu}{\nu-1}|t|^{\frac{\nu(q-1)}{\nu-1}}\\
	& \leq  \frac{2_s^*-q}{2_s^*} w(x)^{\frac{\nu}{\nu-1}\left(\frac{2_s^*}{2_s^*-q}\right)}+\frac{q}{2_s^*}|t|^{\frac{\nu(q-1)}{\nu-1}\frac{2_s^*}{q}},
    \end{align*}
    for a.a.\,$x \in B_\varepsilon = \{x\in\R^n: |x|<R_\varepsilon\}$ and for all $t\in \R$. Let us remark that $$\frac{\nu(q-1)}{\nu-1}\frac{2_s^*}{q}<2_s^*,$$
    since $q<\nu$. Hence the Nemytskii operator $\mathcal{N}_h$ associated with $h(u)=(w|u|^{q-1})^{\frac{\nu}{\nu-1}}$ is continuous from $L^{\frac{\nu(q-1)}{\nu-1}\frac{2_s^*}{q}}(B_\varepsilon)$ in $L^1(B_\varepsilon)$. Then we conclude that
    \begin{align*}
	\int_{|x|<R_\varepsilon} w(x)^{\frac{\nu}{\nu-1}}|u_k|^{\frac{\nu(q-1)}{\nu-1}} \,dx \to \int_{|x|<R_\varepsilon} w(x)^{\frac{\nu}{\nu-1}}|u|^{\frac{\nu(q-1)}{\nu-1}}\,dx,
    \end{align*}
    as $k\to \infty$. Thus, it follows that $w(x)|u_k|^{q-1}$ converges to $w(x)|u|^{q-1}$ in $L^{\frac{\nu}{\nu-1}}(B_\varepsilon)$. Hence since $L^{2_s^*}(B_\varepsilon)\subset L^\nu(B_\varepsilon)$, we have $w(x)|u_k|^{q-1}|u_k|$ converges to $w(x)|u|^{q-1}|u|$ in $L^1(B_\varepsilon)$, that is
    \begin{equation}\label{e3.4}
	\int_{|x|<R_\varepsilon} (w(x)|u_k|^{q-1}|u_k|-w(x)|u|^{q-1}|u|) \,dx \to 0.
    \end{equation}
    Finally, in view of \eqref{e3.3} and \eqref{e3.4}, we get
    \begin{equation}\label{e3.5}
	\int_{\R^N} w(x)|u_k|^{q} \,dx \to \int_{\R^N} w(x)|u|^{q} \,dx.
    \end{equation}
    Now, since $ J(u _k) \geq  S _0 $ we have
    \begin{align*}
	\int_{\R^N} w(x)|u_k|^{q}\,dx
	& \geq  S _0 ( [u _k ]_{s}^{2} +  [u _k ]_{s}^{2_s^*})\\
	& \geq  S _0(\Lambda_1^2 + \Lambda_1^{2_s^*} )  >0.
    \end{align*}
    It follows from (W$_1$) that $ 0< \displaystyle\int_{\R^N} w(x)|u|^{q}\,dx $, therefore $u\not \equiv 0 \;\mbox{in}\;\; \R^N$.

    By Corollary 3.9 of Brezis \cite{Brezis-2011}, from the weak lower semicontinuity of the norm in $D^s\left(\R^N\right)$ and in view of \eqref{e3.5}, we obtain
    \begin{align*}
	\mathcal{S}
	& = \lim \sup _{k \to \infty } J(u _k)\\
	& = \lim \sup _{k \to \infty }  \frac{ \displaystyle\int_{\R^N}  w|u_k|^q\,dx}{ [u _k]_s^2 +  [u _k]_s^{2_s^*}}\\
	& = \lim \sup _{k \to \infty }  \frac { \displaystyle\int_{\R^N}  w|u|^q\,dx}{ [u _k]_s^2 +  [u _k]_s^{2_s^*} }\\
	& \leq \frac {\displaystyle\int_{\R^N}  w|u|^q\,dx }{\displaystyle \lim \inf _{k \to \infty } ([u _k]_s^2 +  [u _k] _s^{2_s^*})}.
    \end{align*}
    Hence, we get
    \begin{eqnarray*}
	\mathcal{S} \leq \frac {\displaystyle\int_{\R^N} w|u|^q\,dx }{[u]_s^2 +  [u]_s^{2_s^*}} = J(u),
    \end{eqnarray*}
    and consequently, we have
    \begin{align*}
	J(u) =\mathcal{S}.
    \end{align*}

    Let $ \varphi$ a fixed element in $D^s\left(\R^N\right) $ and consider
    \begin{align*}
	\xi (\varepsilon ) = [u +\varepsilon \varphi]_s=\left(\ \iint\limits_{\R^{2N}}\frac{|u(x)-u(y)+\varepsilon (\varphi (x)-\varphi (y))|^2}{|x-y|^{N+2s}}dxdy\right)^{\frac{1}{2}}.
    \end{align*}
    Due to the fact that  $ u\neq 0$, and the continuity of $ \xi $, we can find $ \varepsilon _0 > 0 $ such that
    \begin{align*}
	[u +\varepsilon \varphi]_s > 0 \quad \text{for all } \varepsilon \in ]- \varepsilon _0, \varepsilon _0[.
    \end{align*}
    Define a function
    \begin{align*}
    \eta (\varepsilon ) = J( u +\varepsilon \varphi )=\frac{\displaystyle\int_{\R^N}  w|u +\varepsilon \varphi|^q\,dx}{[u +\varepsilon \varphi]_s^2 +  [u +\varepsilon \varphi]_s^{2_s^*}}, \quad \varepsilon \in ]- \varepsilon _0 ,\varepsilon _0[.
    \end{align*}
    The functions $\displaystyle\int_{\R^N}  w|u +\varepsilon \varphi|^q\,dx,\; G(\varepsilon)=[ u +\varepsilon \varphi ]_s^2$ and $H(\varepsilon)=[ u +\varepsilon \varphi ]_s^{2_s^*}$ are differentiable with respect to $\varepsilon$, then we have
    \begin{align*}
	G'(\varepsilon) & =2\iint\limits_{\R^{2N}}\frac{[(u +\varepsilon \varphi)(x)-(u +\varepsilon \varphi)(y)][\varphi(x)-\varphi(y)]}{|x-y|^{N+2s}}dxdy,\\
	H'(\varepsilon) & =2_s^*[u]_s^{2_s^*-2}\iint\limits_{\R^{2N}}\frac{[(u +\varepsilon \varphi)(x)-(u +\varepsilon \varphi)(y)][\varphi(x)-\varphi(y)]}{|x-y|^{N+2s}}dxdy,\\
	\eta'(\varepsilon) & = \frac{q(G(\varepsilon) +  H(\varepsilon))\displaystyle\int_{\R^N}w|u +\varepsilon \varphi|^{q-2}(u +\varepsilon \varphi)\varphi\,dx}{(G(\varepsilon)+H(\varepsilon))^2},\\
	& \qquad -\frac{\left(G'(\varepsilon)+H'(\varepsilon)\right)\displaystyle\int_{\R^N} w|u +\varepsilon \varphi|^q\,dx}{(G(\varepsilon)+H(\varepsilon))^2}.
    \end{align*}
    Since zero is a global maximum of the function $ \eta $, one has $ \eta ^{\prime} (0) = 0. $ This implies that
    \begin{align*}
	q(G(0) +  H(0))\displaystyle\int_{\R^N}w|u|^{q-2}u\varphi\,dx-\left(G'(0)+H'(0)\right)\displaystyle\int_{\R^N} w|u|^q\,dx=0.
    \end{align*}
    Therefore,
    \begin{align*}
	G'(0)+H'(0)=\frac{q(G(0) +  H(0))}{\displaystyle\int_{\R^N} w|u|^q\,dx}\displaystyle\int_{\R^N}w|u|^{q-2}u\varphi\,dx.
    \end{align*}
    Finally, we get
    \begin{align*}
    \iint\limits_{\R^{2N}}\frac{[u(x)-u(y)][\varphi(x)-\varphi(y)]}{|x-y|^{N+2s}}dxdy=\lambda \int_{\R^N} w(x)|u|^{q-2}u\varphi\,dx
    \end{align*}
    for all $\varphi \in D^s\left(\R^N\right)$, where
    \begin{equation}\label{e3.6}
    \lambda = \frac {q ( [u ]_s^2 +  [u ]_s^{2_s^*})}{(2 +  2_s^*[ u]_s^{2_s^*-2}) \displaystyle\int_{\R^N} w(x)|u|^{q} \,dx}.
    \end{equation}
    This completes the proof.
\end{proof}

The next result gives, in particular, the boundedness of weak solutions to \eqref{e3.1}.

\begin{theorem} \label{Thm2}
    Under the assumptions of Theorem \ref{Thm1}, every weak solution $u \in D^s\left(\R^N\right)$ satisfies $ u\in  L^{\kappa}\left(\R^N\right)$ for any  $\kappa \in [ 2_s^*, \infty]$.
\end{theorem}

\begin{proof}
    For every $0<\varepsilon \ll 1$, we define the smooth convex Lipschitz function
    \begin{align*}
	f_\varepsilon(t)= (\varepsilon^2+t^2)^{\frac{1}{2}}.
    \end{align*}
    Taking $\varphi = \psi f_\varepsilon'(u)$ as test function in \eqref{e3.2} with $\psi \in \mathcal{C}_0^{\infty}\left(\R^N\right)$ being a positive function gives
    \begin{align*}
    \iint\limits_{\R^{2N}}\frac{[u(x)-u(y)][\psi f_\varepsilon'(u)(x)-\psi f_\varepsilon'(u)(y)]}{|x-y|^{N+2s}}dxdy=\lambda \int_{\R^N} w(x)|u|^{q-2}u\psi f_\varepsilon'\,dx.
    \end{align*}
    Then using Lemma \ref{Lemma1} with $p=2$, we obtain
    \begin{align*}
	\iint\limits_{\R^{2N}}\frac{[f_\varepsilon(u(x))-f_\varepsilon(u(y))][\psi (x)-\psi (y)]}{|x-y|^{N+2s}}dxdy\leq \lambda \int_{\R^N} |w||u|^{q-1}\psi |f_\varepsilon'|\,dx.
    \end{align*}
    Since $f_\varepsilon$ converges to $f(t)=|t|$, and using the fact that $|f_\varepsilon'(t)|\leq 1$, then by Fatou's Lemma, we get
    \begin{equation}\label{e3.7}
	\iint\limits_{\R^{2N}}\frac{[|u(x)|-|u(y)|][\psi (x)-\psi (y)]}{|x-y|^{N+2s}}dxdy\leq \lambda \int_{\R^N} |w||u|^{q-1}\psi\,dx,
    \end{equation}
    for all $\psi \in \mathcal{C}_0^{\infty}\left(\R^N\right)$. By density, the last inequality remains true for all $\psi \in D^s\left(\R^N\right)$ with $\psi \geq 0$.

    For $ M>0, $ we define the cutoff function  $ u _M(x) = \inf \{u(x),M\}$ of $u$. Then, for $k > 0$ let us choose $ \psi = u _M^{2k + 1} $ as a test function in \eqref{e3.7}. Note that $ \psi  \in D^s\left(\R^N\right)\cap L ^\infty(\R^N)$. This yields
    \begin{align*}
	\iint\limits_{\R^{2N}}\frac{[|u(x)|-|u(y)|][u _M^{2k + 1}(x)-u _M^{2k + 1}(y)]}{|x-y|^{N+2s}}dxdy\leq \lambda \int_{\R^N} |w||u|^{q-1}u _M^{2k + 1}\,dx.
    \end{align*}
    Let $ g(t)=(\inf\{t, M\})^{2k+1}$ to have $G(t)=\int_0^t g'(\tau)^{\frac{1}{2}}\,d\tau, t\in \R$. Then in view of Lemma \ref{Lemma2} with $p=2$ we get
    \begin{align*}
	\frac{2k + 1}{(k+1)^2}\iint\limits_{\R^{2N}}\frac{|u _M^{k + 1}(x)-u _M^{k + 1}(y)|^2}{|x-y|^{N+2s}}dxdy\leq \lambda \int_{\R^N} |w||u|^{q-1}u _M^{2k + 1}\,dx.
    \end{align*}
    Now, in view of \eqref{e2.1}, we obtain for a positive constant $c_4$
    \begin{equation}\label{e3.8}
	\frac{c_4(2k + 1)}{(k+1)^2}\left(\int_{\R^N}u_M^{(k+1)2_s^*}\,dx\right)^{\frac{2}{2_s^*}}\leq \lambda \int_{\R^N} |w||u|^{q-1}u _M^{2k + 1}\,dx.
    \end{equation}
    On the other hand, using (W$_1$), (W$_2$) and H\"older's inequality we derive
    \begin{align*}
	\int_{\R^N} |w||u|^{q-1}u _M^{2k + 1}\,dx
	& \leq \int_{\R^N} |w||u|^{2(k + 1)+q-2}\,dx \\
	& \leq \left(\int_{\R^N} |w|^\eta\,dx\right)^{\frac{1}{\eta}}\left(\int_{\R^N} |u|^{\eta'2(k + 1)}|u|^{\eta'(q -2)}\,dx\right)^{\frac{1}{\eta'}}\\
	& \leq c_5\left(\int_{\R^N} |u|^{r(k+1)}\,dx\right)^{\frac{2}{r}}\left(\int_{\R^N} |u|^{(\frac{r}{2\eta'})'\eta'(q-2)}\,dx\right)^{\frac{1}{\eta'(\frac{r}{2\eta'})'}},
    \end{align*}
    where $c_5$ is a positive constant, $\eta=\frac{\nu}{\nu-1}\frac{2_s^*}{2_s^*-q}, \eta'$ its conjugate exponent and $r=\frac{22_s^*\eta'}{2_s^*-(q-2)\eta'}$. Let us remark that $1<r<2_s^*$ and $(\frac{r}{2\eta'})'\eta'(q-2)=2_s^*$. It follows that
    \begin{equation}\label{e3.9}
	\int_{\R^N} |w||u|^{q-1}u _M^{2k + 1}\,dx \leq c_6\left(\int_{\R^N} |u|^{r(k+1)}\,dx\right)^{\frac{2}{r}},
    \end{equation}
    for some positive constant $c_6$. Hence, combining \eqref{e3.8} and \eqref{e3.9}, there exists a constant $ c_7 > 0 $ independent of  $ M > 0 $ and $ k > 0 $ such that
    \begin{align*}
	\left(\int_{\R^N}u_M^{(k+1)2_s^*}\,dx\right)^{\frac{2}{2_s^*}}\leq c_7 \frac{(k+1)^2}{(2k + 1)}\left(\int_{\R^N} |u|^{r(k + 1)}\,dx\right)^{\frac{2}{r}},
    \end{align*}
    that is
    \begin{equation}\label{e3.10}
	\| u _M \| _{(k + 1)2_s^*} \leq c_8^\frac{1}{k + 1} \left[ \frac{k + 1}{(2k + 1)^{\frac{1}{2}}}\right] ^{\frac {1}{k + 1}}\| u \| _{(k + 1)r},
    \end{equation}
    with $ c_8 = c_7^\frac{1}{2}$.

    Since $ u\in D^s\left(\R^N\right)$, hence $ u\in L^{2_s^*}\left(\R^N\right)$, one can choose  $ k _1 $ in \eqref{e3.10} such that  $ (k _1  + 1 )r = 2_s^* $, that is $ k _1 = \frac{2_s^*}{r} - 1$. Thus
    \begin{align*}
	\| u _M\| _{(k _1 + 1)2_s^*} \leq c_8^\frac {1}{k _1 + 1} \left[\frac {k _1 + 1}{(2k _1 + 1)^\frac {1}{2}}\right]^\frac {1}{k _1 + 1} \| u \| _{2_s^*} \quad \text{for all }M > 0.
    \end{align*}
    Note that $ \lim _{M\to \infty} u _M(x) = u(x)$. Then, Fatou's lemma implies
    \begin{align*}
	\| u \| _{(k _1 + 1)2_s^*} \leq c_8^\frac {1}{k _1 + 1} \left[\frac {k _1 + 1}{(2k _1 + 1)^\frac {1}{2}}\right]^\frac {1}{k _1 + 1} \| u \| _{2_s^*}.
    \end{align*}
    Therefore  $ u\in L^{(k _1 + 1)2_s^*}\left(\R^N\right)$. By the same argument, one can choose $ k _2 $ in \eqref{e3.10} such that  $  (k _2  + 1 )r =\left (k _1  + 1 \right)2_s^* $, that is $ k _2 = \left(\frac{2_s^*}{r}\right)^2 - 1$ to have
    \begin{align*}
	\| u \| _{(k _2 + 1)2_s^*} \leq c_8^\frac {1}{k _2 + 1} \left[\frac {k _2 + 1}{(2k _2 + 1)^\frac {1}{2}}\right]^\frac {1}{k _2 + 1} \| u \| _{(k _1 + 1)2_s^*}.
    \end{align*}
    By iteration, we obtain  $ k _n = (\frac{2_s^*}{r})^n - 1 $ such that
    \begin{align*}
	\| u \| _{(k _n + 1)2_s^*} \leq c_8^\frac {1}{k _n + 1} \left[\frac {k _n + 1}{(2k _n + 1)^\frac {1}{2}}\right]^\frac {1}{k _n + 1} \| u \| _{(k _{n - 1} + 1)2_s^*} \quad \text{for all }n\in\N.
    \end{align*}
    It follows
    \begin{align*}
	\| u \| _{(k _n + 1)2_s^*} \leq c_8^{\sum _{i = 1}^n\frac {1}{k _i + 1}} \prod _{i = 1}^n \left[\frac {k _i + 1}{(2k _i + 1)^\frac {1}{2}}\right]^\frac {1}{k _i + 1} \| u \| _{2_s^*},
    \end{align*}
    or equivalently
    \begin{align*}
	\| u \| _{(k _n + 1)2_s^*} \leq c_8^{\sum _{i = 1}^n\frac {1}{k _i + 1}} \prod _{i = 1}^n \left[\left[\frac {k _i + 1}{(2k _i + 1)^\frac {1}{2}}\right]^\frac {1}{\sqrt{k _i + 1}}\right]^{\frac {1}{\sqrt{k _i + 1}}} \| u \| _{2_s^*}.
    \end{align*}
    Since
    \begin{align*}
	\left[\frac {z + 1}{(2z + 1)^\frac {1}{2}}\right]^\frac {1}{\sqrt{z + 1}} > 1 \quad \text{for all }z > 0 \quad \mbox{and} \quad \lim _{z\to \infty } \left[\frac {z + 1}{(2z + 1)^\frac {1}{2}}\right] ^\frac {1}{\sqrt{z + 1}} = 1,
    \end{align*}
    there exists a constant $ c_9 > 0 $ independent of $ n \in {\N} $ such that
    \begin{align*}
    \| u \| _{(k _n + 1)2_s^*} \leq c_8^{\sum _{i = 1}^n\frac {1}{k _i + 1}} c_9^{\sum _{i = 1}^n\frac {1} {\sqrt{k_i + 1}}}\| u \| _{2_s^*},
    \end{align*}
    where
    \begin{align*}
	\frac {1}{k _i + 1} = \left(\frac {r}{2_s^*}\right)^i, \quad \frac {1}{\sqrt {k _i + 1}} = \left(\sqrt  {\frac {r}{2_s^*}}\right)^i, \quad \frac {r}{2_s^*} < \sqrt {\frac {r}{2_s^*}} < 1.
    \end{align*}
    Hence, there exists a constant  $ c_{10} > 0 $  independent of  $ n \in {\N} $ such that
    \begin{equation}\label{e3.11}
	\| u \| _{(k _n + 1)2_s^*} \leq c_{10}\| u \| _{2_s^*}
    \end{equation}
    for all $n \in {\N}$. Passing to the limit as $n$ goes to infinity we get
    \begin{equation}\label{e3.12}
	\| u \| _\infty \leq c_{10}\| u \| _{2_s^*}.
    \end{equation}
    Therefore, due to \eqref{e3.11} and \eqref{e3.12}, we deduce that
    \begin{align*}
    u\in L^{\kappa}\left(\R^N\right) \quad \mbox{for all} \quad 2_s^* \leq \kappa \leq \infty.
    \end{align*}
\end{proof}

\begin{proposition}
    Let $u\in D^s\left(\R^N\right)\backslash \{0\}$ be a weak solution of problem \eqref{e3.1} corresponding to the parameter $\lambda$ given in \eqref{e3.6}. Then
    \begin{align*}
	[u]_s=\left(\frac{q-2}{2_s^*-q}\right)^{\frac{1}{2_s^*-2}}.
    \end{align*}
\end{proposition}
\begin{proof}
    Let $u\in D^s\left(\R^N\right)\backslash \{0\}$ be a weak solution of problem \eqref{e3.1}, that is,
    \begin{align*}
	\iint\limits_{\R^{2N}}\frac{[u(x)-u(y)][\varphi(x)-\varphi(y)]}{|x-y|^{N+2s}}dxdy=\lambda \int_{\R^N} w(x)|u|^{q-2}u\varphi\,dx,
    \end{align*}
    for all $\varphi \in D^s\left(\R^N\right)$. Taking $\varphi = u$ it follows that
    \begin{align*}
	[u]_s^2
	& = \lambda \int_{\R^N} w(x)|u|^{q}\,dx \\
	& = \frac{q \left( [u ]_s^2 +  [u ]_s^{2_s^*}\right)}{\left(2 +  2_s^*[ u]_s^{2_s^*-2}\right)\int_{\R^N} w(x)|u|^{q} \,dx} \int_{\R^N} w(x)|u|^{q} \,dx. \\
    \end{align*}
    Hence, a simple calculation gives the result.
\end{proof}

\begin{remark}\label{Rem3.7}
    Consider the following more general class of nonlinear and nonlocal eigenvalue problems
    \begin{equation}\label{e3.13}
	\left(-\Delta_p\right)^s u = \lambda f(x, u) \quad\text{in}\quad {\R}^N,
    \end{equation}
    where
    \begin{align*}
	(-\Delta_p)^s u(x)=2\lim_{\varepsilon \to 0^+}\int_{\{y\in\R^N:|y-x|\geq \varepsilon\}}\frac{|u(x)-u(y)|^{p-2}(u(x)-u(y))}{|x-y|^{N+sp}}dy,
    \end{align*}
    is the fractional p-Laplacian, $0<\lambda$, $0<s<1$ and $sp< N$.
    Assume that $f$ is a Carath\'eodory function satisfying
    \begin{enumerate}
	\item [(H$_1$)]
	    There exists an open subset $\Omega \neq \emptyset$ of $\R^N$ such that $f(x,t)>0$ for a.a.\,$x\in\Omega$ and for all $t\in \R$;
	\item [(H$_2$)]
	    $|f(x,t)|\leq g(x)t^{q-1}$ for all $x\in \R^N $ and for all $t\geq 0$ with $p<q<p_s^*$, where $g$ is nonnegative function such that
	    \begin{align*}
		g\in L^{\frac{p_s^*}{p_s^*-q}}\left(\R^N\right)\cap L^{\frac{p_s^*}{p_s^*-q}+\delta}\left(\R^N\right)
	    \end{align*}
	    for some constant $ \delta \in [0,\infty)$ and $p_s^*=\frac{pN}{N-sp}$.
    \end{enumerate}
    To deal with the problem \eqref{e3.13}, we shall use the following function space
    \begin{align*}
	D^{s,p}\left(\R^N\right)=\left\{ u \in L^{p_s^*}\left(\R^N\right): \iint\limits_{\R^{2N}}\frac{|u(x)-u(y)|^p}{|x-y|^{N+sp}}dxdy<\infty \right\},
    \end{align*}
    which is a Banach space under the norm
    \begin{align*}
    [u]_{s,p}= \left(\ \iint\limits_{\R^{2N}}\frac{|u(x)-u(y)|^p}{|x-y|^{N+sp}}dxdy\right)^{\frac{1}{p}}.
    \end{align*}
    Here solutions of \eqref{e3.13} are always understood in the weak sense, that is,
    $u \in D^{s,p}\left(\R^N\right)$ is a weak solution of problem \eqref{e3.13} associated to the eigenvalue $\lambda >0$ if
    \begin{align*}
	\iint\limits_{\R^{2N}}\frac{|u(x)-u(y)|^{p-2}[u(x) - u(y)]}{|x-y|^{N+sp}}[\varphi(x)-\varphi(y)]dxdy= \lambda\int_{\R^N} f(x,u)\varphi\,dx,
    \end{align*}
    is satisfied for every $\varphi \in D^{s,p}\left(\R^N\right)$.

    One can prove that similar assertions as stated in Theorems \ref{Thm1} and \ref{Thm2} remain true concerning problem \eqref{e3.13} by the use of the same approach and arguments similar to those used in their proofs.
\end{remark}

\begin{remark}\label{Rem3.8}
    Another generalization of problem \eqref{e3.1} that is of interest can be given when $(-\Delta)^s$ is replaced by any nonlocal integro-differential operator
    \begin{align*}
    \mathcal{L_K} u(x)=-\frac{1}{2}\int_{\R^N} [u(x+y)+u(x-y)-2u(x)]K(y)\,dy \quad x\in \R^N,
    \end{align*}
    along any rapidly decaying function $u$ of class ${\mathcal{C}}^\infty\left(\R^N\right)$, where the positive weight $K: \R^N\backslash \{0\} \to \R^+$ satisfies the following properties:
    \begin{itemize}
	\item[(K$_1)$]
	    $\nu K \in L^1\left(\R^N\right)$ with $\nu(x)=\min(|x|^2, 1)$;
	\item[(K$_2)$]
	    there exists $\Gamma >0$ such that $K(x)\geq \Gamma |x|^{-(N+2s)}$ for any $x\in \R^N\backslash \{0\}$;
	\item[(K$_3)$]
	    $K(x)=K(-x)$ for any $x\in \R^N\backslash \{0\}$.
    \end{itemize}
    In particular, when the kernel $K$ is given by $K(x)= |x|^{-(N+2s)}$, $\mathcal{L_K}$ coincides with the fractional Laplace operator $(-\Delta)^s$.\\
    If $2s <N$ (where $s \in(0, 1)$), then under an adequate functional framework and suitable conditions on the right-hand side term, the assertion of Theorems \ref{Thm1} and \ref{Thm2} should be a generalization of previous results.
\end{remark}

\section{Nonlinear problems with two parameters}

We consider the following nonlinear problem with two parameters
\begin{equation} \label{e4.1}
    (-\Delta)^s u = \lambda f(x,u)+\mu g(x,u)   \quad \mbox{in } {\R}^N,
\end{equation}
where $ 0<s<1$, $2s< N$, $2< q<\min(r, 2_s^*)$, $\lambda, \mu>0$ are real parameters, $f,g: \R^N \times \R \to \R$ are two nonlinearities that have subcritical growth with respect to t and $(-\Delta)^s$ is the fractional Laplacian operator defined in \eqref{e1.1}. To be more precise, we assume that $f$ and $g$ are Carath\'eodory functions satisfying the following conditions.
\begin{enumerate}
    \item[(F)]
	$|f(x,t)| \leq m(x)|t|^{q-1}$ for a.a.\,$x\in \R^N$ and for all $t\in \R$, where $m$ is a positive function such that $m\in L^{\frac{2_s^*}{2_s^*-1}}\left(\R^N\right) \cap L^{\frac{2_s^*}{2_s^*-q}+\delta}\left(\R^N\right)$ with some $\delta>0$;
    \item[(G)]
	$|g(x,t)| \leq h(x)|t|^{r-1}$ for a.a.\,$x\in \R^N$ and for all $t\in \R$ for some positive $h$ such that $h\in L^{\frac{2_s^*}{2_s^*-r}}\left(\R^N\right)\cap L^{\frac{2_s^*}{2_s^*-r}+\delta}\left(\R^N\right)$ if $1<r<2_s^*$\\
	or $|g(x,t)| \leq h(x)|t|^{2_s^*-1}$ for a.a.\,$x\in \R^N$ and for all $t\in \R$ with $h\in L^\infty\left(\R^N\right)$ and $g(x,0)=0$ if $r=2_s^*$.
\end{enumerate}

Recall that a weak solution of problem \eqref{e4.1} is any $u \in D^s\left(\R^N\right)$ such that
\begin{equation}\label{e4.2}
    \iint\limits_{\R^{2N}}\frac{[u(x)-u(y)][\varphi(x)-\varphi(y)]}{|x-y|^{N+2s}}dxdy= \int_{\R^N} (\lambda f(x,u)+\mu g(x,u))\varphi\,dx,
\end{equation}
for all $\varphi \in D^s\left(\R^N\right)$.

Now we state our main result of this section about the multiplicity of solutions to problem \eqref{e4.1}.
\begin{theorem}\label{Thm4.1}
    Let $0<s<1$ with  $2s< N$ and assume (F), (G). Furthermore, suppose, for $1<\tau<2$ and some positive function $\rho \in L^{\frac{2_s^*}{2_s^*-\tau}}\left(\R^N\right)$, that
    \begin{equation} \label{e4.3}
	\limsup\limits_{|u|\to
    +\infty }\frac{F(x,u)}{\rho(x)|u|^\tau}\leq M<+\infty \quad \mbox{uniformly for a.a.\,}x\in {\R}^N
    \end{equation}
    and
    \begin{align*}
	\sup_{u\in D^{s}({\R}^N)}\int_{{\R}^N}F(x,u)dx>0
    \end{align*}
    where $F(x,t)=\int_0^tf(x,\xi)\,d\xi$. Set
    \begin{equation} \label{e4.4}
    \theta=\frac{1}{2}\inf\left\{\frac{\displaystyle\iint\limits_{\R^{2N}}\frac{|u(x)-u(y)|^2}{|x-y|^{N+2s}}dxdy}{\displaystyle\int_{{\R}^N}F(x,u)dx}: u\in D^{s}\left(\R^N\right), \int_{{\R}^N}F(x,u)dx>0\right\}.
    \end{equation}
    Then, for each compact interval $[a, b]\subset \left ]\theta, +\infty \right[$, there exists $r_1>0$ with the following property: for every $\lambda \in [a, b]$, there exists $\delta > 0$ such that for each $\mu \in [0, \delta]$, the equation \eqref{e4.1} has at least two nonzero weak solutions in $D^{s}\left(\R^N\right)$ whose norms are less than $r_1$.
\end{theorem}
\begin{proof}
    The  functional $\Phi:D^{s}\left(\R^N\right) \to \R $ defined as
    \begin{align*}
    \Phi (u)= \frac{1}{2}\iint\limits_{\R^{2N}}\frac{|u(x)-u(y)|^2}{|x-y|^{N+2s}}dxdy
    \end{align*}
    is well defined and coercive. By Lemma 3.2 of Autuori-Pucci \cite{Autuori-Pucci-2013} we know that $\Phi$ is continuous, G\^ateaux differentiable and belongs to $C^1(\left(D^{s}\left(\R^N\right), \R\right)$ with
    \begin{align*}
	\langle \Phi'(u),\varphi\rangle= \iint\limits_{\R^{2N}}\frac{[u(x)-u(y)][\varphi(x)-\varphi(y)]}{|x-y|^{N+2s}}dxdy,
    \end{align*}
    for all $u, \varphi \in D^{s}\left(\R^N\right)$. Moreover, $\Phi$ is weakly lower semicontinuous by Corollary 3.9 of Brezis \cite{Brezis-2011} and it is clearly bounded on each bounded subset of $D^{s}\left(\R^N\right)$. Finally, by using Theorem 26.A of Zeidler \cite{Zeidler-1990}, we deduce that $\Phi'$ admits a continuous inverse on $\left(D^{s}\left(\R^N\right)\right)^*$. This implies that $\Phi \in W_{D^{s}\left(\R^N\right)}$. We define the functional $J:D^{s}\left(\R^N\right)\to \R$ by
    \begin{align*}
	J(u)=\int_{\R^N} F(x,u)\,dx.
    \end{align*}
    From (F) it follows that
    \begin{align*}
    |F(x,t)|\leq \frac{1}{q}m(x)|u|^{q} \quad \text{for a.a.\,}x\in\R^N \text{ and for all }t\in\R.
    \end{align*}
    Then, by standard argument, $F$ is in $C^1\left(\R^N\times \R\right)$, hence we see that $J$ is well defined and continuously G\^ateaux differentiable with
    \begin{align*}
	J^{'}(u) v = \int_{\R^N} f(x, u)v \,dx \quad \text{for all }u,v\in D^{s}\left(\R^N\right).
    \end{align*}

    Let us prove that $J'$ is a compact map from $D^{s}\left(\R^N\right)$ to $(D^{s}\left(\R^N\right))^*$.
    Indeed, let $(u_k)_{k \geq 1}$ be a sequence in $D^{s}\left(\R^N\right)$ which converges weakly to $u$. On one hand, in view of H\"older's inequality, \eqref{e2.1} and since $\frac{2_s^*}{2_s^*-1} < \frac{2_s^*}{2_s^*-q}$, we obtain
    \begin{align*}
	\int_{|x|\geq R}f(x,u)v\,dx\leq C_{2_s^*}^q\left(\int_{|x|\geq R}|m|^{\frac{2_s^*}{2_s^*-q}}\,dx\right)^{\frac{2_s^*-q}{2_s^*}}[u]_s^{q-1}[v]_s,
    \end{align*}
    for all $u, v\in D^{s}\left(\R^N\right)$ and for all $0\leq R\leq +\infty$. Since $(u_k)_{k \geq 1}$ is a bounded sequence, this implies for any $\varepsilon>0$ the existence of $R_\varepsilon >0$ such that
    \begin{equation} \label{e4.5}
	\int_{|x|\geq R_\varepsilon} f(x, u_k)v \,dx \leq \varepsilon \quad \text{and}\quad \int_{|x|\geq R_\varepsilon} f(x, u)v \,dx \leq \varepsilon
    \end{equation}
    holds for all $k$.

    On the other hand, let $\theta_1=1+\frac{2_s^*-q}{2_s^*}\delta$. Applying Young's inequality gives
    \begin{align*}
	|f(x,t)|^{\theta_1}
	& \leq m(x)^{\theta_1}|t|^{(q-1)\theta_1}\\
	& \leq \frac{2_s^*-q}{2_s^*} m(x)^{\theta_1 \frac{2_s^*}{2_s^*-q}}+\frac{q}{2_s^*}|t|^{(q-1)\theta_1\frac{2_s^*}{q}},
    \end{align*}
    for a.a.\,$x \in B_\varepsilon = \{x\in\R^N: |x|<R_\varepsilon\}$ and for all $t\in \R$. Note that
    \begin{align*}
	\theta_1 \frac{2_s^*}{2_s^*-q} = \frac{2_s^*}{2_s^*-q}+\delta,\quad (q-1)\theta_1\frac{2_s^*}{q}<2_s^* \quad \mbox{and}\quad \theta_1'<2_s^*
    \end{align*}
    for some $\delta>0$ such that $\frac{2_s^*}{(2_s^*-q)(2_s^*-1)}<\delta < \frac{2_s^*}{(2_s^*-q)(q-1)}$ with $\theta_1'$ being the conjugate exponent of $\theta_1$.\\
    Hence, in view of Lemma \ref{Lemma3}, we conclude that the Nemytskii operator $\mathcal{N}_{f^{\theta_1}}$ associated with $f^{\theta_1}$ is continuous from $L^{(q-1)\theta_1\frac{2_s^*}{q}}(B_\varepsilon)$ in $L^1(B_\varepsilon)$ and that
    \begin{align*}
    \int_{|x|<R_\varepsilon} f^{\theta_1}(x,u_k) \,dx \to \int_{|x|<R_\varepsilon} f^{\theta_1}(x,u) \,dx.
    \end{align*}
    This implies that $f^{\theta_1}(x,u_k)$ converges to $f^{\theta_1}(x,u)$ in $L^1(B_\varepsilon)$. Taking $L^{2_s^*}(B_\varepsilon)\subset L^{\theta_1'}(B_\varepsilon)$ into account, we get that $f(x,u_k)v$ converges to $f(x,u)v$ in $L^1(B_\varepsilon)$, that is
    \begin{equation}\label{e4.6}
	\int_{|x|<R_\varepsilon} (f(x,u_k)-f(x,u))v \,dx \to 0,
    \end{equation}
    for all $v\in D^s\left(\R^N\right)$. Finally, in view of \eqref{e4.5} and \eqref{e4.6}, there holds
    \begin{align*}
	\int_{\R^N} f(x, u_k)v \,dx \to \int_{\R^N} f(x, u)v \,dx.
    \end{align*}
    Therefore $J'$ is a compact operator.

    We claim that
    \begin{align*}
	\limsup_{u\to  0}\frac{J(u)}{\Phi (u)}\leq 0.
    \end{align*}
    For all $u\in D^{s}\left(\R^N\right)\backslash\{0\}$ we have
    \begin{align*}
	\frac{J(u)}{\Phi (u)}\leq \frac{2}{q}\frac{\displaystyle\int_{\R^N}m|u|^{q}\,dx}{[u]_s^2}\leq \frac{2}{q}\frac{\|m\|_\frac{2_s^*}{2_s^*-q}\|u\|_{2_s^*}^{q}}{[u]_s^2}.
    \end{align*}

    Then, using again \eqref{e2.1}, we obtain
    \begin{align*}
    \frac{J(u)}{\Phi (u)} \leq  \frac{2}{q}C_{2_s^*}^q\frac{\|m\|_\frac{2_s^*}{2_s^*-q}[u]_s^q}{[u]_s^2}.
    \end{align*}
    Consequently, since $2<q$ we infer for $\varepsilon > 0$ small enough
    \begin{equation} \label{e4.7}
	\limsup_{u\to  0}\frac{J(u)}{\Phi (u)}\leq \frac{2}{q}C_{2_s^*}^q\|m\|_\frac{2_s^*}{2_s^*-q}\varepsilon.
    \end{equation}

    Next, let us prove that
    \begin{align*}
	\limsup_{[u]_s\to  \infty}\frac{J(u)}{\Phi (u)}\leq 0.
    \end{align*}
    Indeed, using \eqref{e4.3}, there exists $A>0$ such that
    \begin{align*}
	|F(x,u)|\leq   M\rho(x)|u|^\tau \quad \text{for all } |u|>A,
    \end{align*}
    where $\tau<2$ and $M>0$. Then, using H\"older's inequality and \eqref{e2.1}, we obtain for each $u\in D^{s}\left(\R^N\right)\backslash\{0\}$
    \begin{align*}
	\frac{J(u)}{\Phi (u)}
	& \leq \frac {2\int_{\R^N (|u|> A)} F(x,u)\,dx}{[u]_s^2} + \frac {2\int_{\R^N (|u|\leq A)}F(x,u)\,dx}{[u]_s^2}\\
	& \leq \frac {2M\int_{\R^N (|u|> A)}\rho(x)|u|^\tau \,dx}{[u]_s^2} + \frac{2}{q} \frac{\int_{\R^N (|u|\leq A)}m(x)|u|^{q}dx}{[u]_s^2}\\
	& \leq \frac{2MC_{2_s^*}^\tau\|\rho \|_{L^{\frac{2_s^*}{2_s^*-\tau}}}[u]_s^\tau}{[u]_s^2} + \frac{2A^{q-1}}{q} \frac{\int_{\R^N (|u|\leq A)}m(x)|u|dx}{[u]_s^2}.
    \end{align*}
    Since $m\in L^{\frac{2_s^*}{2_s^*-1}}\left(\R^N\right)$ and by using \eqref{e2.1}, we have
    \begin{align*}
	\frac{J(u)}{\Phi (u)} \leq \frac{2MC_{2_s^*}^\tau\|\rho \|_{L^{\frac{2_s^*}{2_s^*-\tau}}}[u]_s^\tau}{[u]_s^2} + \frac{2A^{q-1}C_{2_s^*}}{q} \frac{\|m \|_{L^{\frac{2_s^*}{2_s^*-1}}}[u]_s}{[u]_s^2}.
    \end{align*}
    Consequently, this yields
    \begin{align*}
    \frac{J(u)}{\Phi (u)} \leq \frac{2MC_{2_s^*}^\tau\|\rho \|_{L^{\frac{2_s^*}{2_s^*-\tau}}}}{[u]_s^{2-\tau}} + \frac{2A^{q-1}C_{2_s^*}}{q} \frac{\|m \|_{L^{\frac{2_s^*}{2_s^*-1}}}}{[u]_s}.
    \end{align*}
    Therefore, since $\tau<2$ we obtain
    \begin{equation} \label{e4.8}
    \limsup_{[u]_s \to +\infty}\frac{J(u)}{\Phi (u)}\leq 2C_{2_s^*}\left(MC_{2_s^*}^{\tau-1}\|\rho\|_{L^{\frac{2_s^*}{2_s^*-\tau}}}+\frac{pA^{q-1}}{q}\|m\|_{L^\frac{p^*}{p^*-1}}\right)\varepsilon.
    \end{equation}
    Finally, in view of \eqref{e4.7} and \eqref{e4.8} and since $\eps>0$ is arbitrary chosen, that
    \begin{align*}
	\max \left\{\limsup_{[u]_s\to +\infty} \frac{J(u)}{\Phi (u)},\limsup_{u\to 0} \frac{J(u)}{\Phi (u)}\right\}\leq 0.
    \end{align*}
    Now the proof of Theorem \ref{Thm4.1} will be accomplished by applying Theorem \ref{ThmRic}. Indeed all the assumptions of Theorem \ref{ThmRic} are satisfied with $\alpha=0$ and
    $\beta=\frac{1}{\theta}$, where $\theta$ is as in \eqref{e4.4} and choose $[a,b]\subseteq ]\theta,+\infty[$.
    Moreover, the functional
    \begin{align*}
	\Psi (u)= \frac {1}{r}\int_{\R^N}h(x)|u|^{r}\,dx
    \end{align*}
    is well defined and continuously G\^ateaux differentiable in $D^s\left(\R^N\right)$ and one has
    \begin{align*}
	\Psi'(u)\varphi = \int_{\R^N}h(x)|u|^{r-2}u\varphi\,dx \quad \text{for all }u, \ph\in D^s\left(\R^N\right).
    \end{align*}
    As shown for $J$, we can prove in the same way that $\Psi'$ is compact. Therefore, there exists $r>0$ such that for every $\lambda\in[a,b]$ we can find $\sigma > 0$ verifying the following condition: for each $\mu \in [0, \sigma]$, the functional $\Phi - \lambda J - \mu \Psi$ has at least three critical points, which are precisely weak solutions of problem $\eqref{e4.1}$ whose norms are less than $r$. This yields the assertion of the theorem.
\end{proof}

\begin{proposition}
  Let the assumptions of Theorem \ref{Thm4.1} be satisfied. Then, for any solution $u$ of problem \eqref{e4.1} there holds $u\in L^\sigma\left(\R^N\right)$ for all $ 2_s^*\leq\sigma \leq \infty$.
\end{proposition}

The proof follows by adapting the same iteration scheme as in the proof of Theorem \ref{Thm2}. Few details, due the condition assumed on $h$, may be given to bound the maximal norm of solutions.\\

\begin{remark} ~
\begin{enumerate}[leftmargin=0.8cm]
    \item[(i)]
	Let us remark that in the proof of Theorem \ref{Thm4.1}, the choice of the values of the positive constant $\delta$ does not affect on the condition imposed on the function $f$.
    \item[(ii)]
	Problems like \eqref{e4.1} can be generalized for a more general class of fractional operators. Namely, by considering problems driven from the fractional p-Laplacian defined in Remark \ref{Rem3.7}, the assertion in Theorem \ref{Thm4.1} holds true under some suitable conditions and variational setting.
\end{enumerate}
\end{remark}

\subsection*{Acknowledgements}
The first author gratefully acknowledges the financial support of the German Academic Exchange Service (DAAD).

\end{document}